\documentclass{article}

\usepackage{latexsym}
\usepackage{amsmath,amsthm,amssymb}
\usepackage{hyperref} %% for PDF

\theoremstyle{plain}
\newtheorem*{hypothesis}{Hypothesis}%{Conjecture} 
\newtheorem{theorem}{Theorem}[section]

\newtheorem{proposition}[theorem]{Proposition}
\newtheorem{corollary}[theorem]{Corollary}

\theoremstyle{definition}
\newtheorem{definition}[theorem]{Definition}
\newtheorem{example}[theorem]{Example}

\theoremstyle{remark}
\newtheorem{remark}[theorem]{Remark}
\renewenvironment{proof}{\noindent{\it Proof}.}{\qed}

\usepackage{amsmath}
\DeclareMathOperator{\ord}{ord}
\DeclareMathOperator{\ind}{ind}
\title{Evaluation of the Effectiveness of the Frobenius Primality Test}
\author{Sergei Khashin 
}

\date{}
\begin{document}
\maketitle
%\Large
\begin{abstract}
The Frobenius primality test is based on the properties of the Frobenius automorphism of the quadratic extension of the residue field. Although it is probabilistic, we show that is ``very rarely wrong''. To date there are no counterexamples to this method  and there are reasons to believe that they do not exist at all. In this paper, we suggest a version of the Frobenius test and 
prove   that it does not fail for numbers less than $2^{64}$. We also show that a ``Frobenius pseudoprime'' will  necessarily have a prime divisor greater than 3000. 
\end{abstract}


%\tableofcontents
% \pagebreak

%------------------------------------------------------------------------------
\section*{Introduction}\label{s1}

The most of popular nowadays methods for primality testing are based on small Fermat theorem: Miller-Rabin and Solovay-Strassen primality tests. However reliability of such methods is not high: for
example in~\cite{Sorenson}, $24$- and $25$-digit numbers are found
that pass twelve and thirteen Miller-Rabin tests, respectively.
Therefore, even a few dozen of 
positive tests applied to some particular number does not not guarantee the primality of that number. This is  important for applications and, for example, in the Java language another test for numbers longer than $100$ bits is also used, the Lucas test, see~\cite{Bail}. That test has a significantly higher reliability, but a mathematical study of the combined use of these tests is difficult. 

The Frobenius primarity test method is based on the Frobenius automorphism of the finite field of order $p^2$, $GF(p^2)$ for some prime $p$.
It has been known
for a long time, see for example,~\cite{GrPom,DamFr2,Gr}. In~\cite{DamFr2, Seysen}, even stronger versions of this test were suggested.
However over the years the Frobenius method was greatly underestimated.

The reason for this is twofold. 
First of all, it is a common belief  that there are some small pseudo-primes for this test. For example, in the book~\cite
[p.$146$]{GrPom} it is stated that the number $5777 = 53\cdot 109$ is
a Frobenius pseudo-prime (FPP) for $c=5$. However, it is easy to verify that this is not the case. Apparently, at this point in the book, the term
``FPP'' is used in a slightly different sense. Secondly, as it was
established in~\cite{DamFr2}, an upper bound on the error probability of the Frobenius method is $\approx 1/1300$. Although this is much
less than the estimate for the Miller-Rabin $(1/4)$ method, 
still the probability error looks very significant.

In the present paper, beside other results, we show that Frobenius method does not fail on numbers less $2^{64}$. In fact, to date 
%there is no counterexample to this method In fact,
no single composite number is known to pass even the simplest version of the test, and it is our hypothesis that FPP do not exist at all.

Frobenius test consists in checking some equality in quadratic extension
of the integers modulo prime $p$. The equality of the norms of the corresponding elements is
equivalent to the Fermat test, and the equality of the irrational
parts is the Lukas test. That is, the Frobenius test is a natural
union of these two tests.

The complexity of the Frobenius test is twice the complexity of the
methods Fermat or Miller-Rabin, that is equal to the complexity of
two such tests.

The Miller-Rabin test for the number $n$ begins with a choice of
the base $a$, which is relatively prime to $n$. As the base, one either takes  the first prime numbers, $2,3,5,\dots$, or makes a pseudo-random choice of the number $a$ that is relatively
prime to $n$. In the usual definition of the Frobenius test (see, for example,~
\cite{DamFr2}), it is also suggested to make a pseudo-random
choice of the ``base''  $z=a+b\sqrt{c}$.

In our approach, we propose to restrict this choice to the forms
$2+\sqrt{c}$ or $1+\sqrt{c}$ depending on $c$ (for details see Definition~\ref{FrobTest}). This is much more convenient and, most
importantly, sufficient.
%quite enough.
Nevertheless, most of the theorems is
given for arbitrary $a$ and $b$. 

The paper is organized as follows.
In Sec.~\ref{notation} we give the
necessary information and fix the notation.
In Sec.~\ref{sec_Frob} 
we introduce Frobenius method (Definition \ref{FrobTest})
and discover its properties. 
In Sec~\ref{sec:calc} we describe the non-trivial approach that lead to an algorithm that will allow to show that the Frobenius method does not fail on the numbers less $2^{64}$. In Sec.~\ref{sec:last} we show that an FPP necessarily has a prime divisor $>3000$.

The computational results of the paper were only possible due to our theoretical results on properties of FPPs. This significantly simplified the number of cases to consider 
and thus allowed to run the computations in some acceptable time.
Statements that require only mathematical reasoning 
are called ``Theorems'',
%(\ref {main_thm}...\ref{thm_last}), 
and
statements that in part require some computer calculations are called
``Propositions''.
%( \ref{prop_1direct}...\ref{prop_64}).
The main results of the paper are
%properties of FPP:
Theorem~\ref{thm:mult_factors}, Theorem~\ref{thm51}, Theorem~\ref{admit_div},
and Theorem~\ref{thm_coprime2}, and also
Proposition \ref{prop_64} 
%claiming that the Frobenius method does not fail on the numbers less
%$2^{64}$ 
and Proposition \ref{small_div}. %claiming that the FPP
%necessarily has a prime divisor $>3000$. 

%===============================================================================
\section{Notations and preliminary information}\label{notation}

\subsection{Jacobi symbol}

We refer the reader to \cite{vinogr} or \cite {vasil} for the definition and main properties of the Jacobi symbol, which we denote by $J(a/n)$. Here is a list of the properties that we shall use. 
(By
$gcd(a,b)$ we denote the greatest common divisor.)
\begin{itemize}
\item $J(a+n/n) =J(a/n)$.
\item If $p$ is prime and $gcd(a,p)=1$, then $J(a/p) = a^{(p-1)/2} \mod p$.
\item $J(ab/n) = J(a/n)J(b/n)$.
\item Let $n$ is odd and $n=n_1n_2$. Then $J(a/n) = J(a/n_1)J(a/n_2)$.
\item Let $p,q$ are odd. Then $J(p/q) = (-1)^{\frac{p-1}{2}\frac{q-1}{2}} J(q/p)$.
\end{itemize}

Below we give the values of the $J(a/n)$ for some $a$   that we shall need in what follows.

\[
 \begin{array}{ll}
 J(-1/n) &= \left\{
\begin{aligned}
 1,  \quad &n \equiv 1 \mod 4 \\
 -1, \quad &n \equiv 3 \mod 4 \\
\end{aligned}
 \right. \\
  J(2/n) &= \left\{
\begin{aligned}
 1,  \quad &n \equiv \pm 1 \mod 8 \\
 -1, \quad &n \equiv \pm 3 \mod 8 \\
\end{aligned}
 \right.\\
  J(3/n)  &= \left\{
\begin{aligned}
 1,  \quad &n \equiv \pm 1 \mod 12 \\
 -1, \quad &n \equiv \pm 5 \mod 12 \\
\end{aligned}
 \right. \quad(gcd(6,n)=1) \\
  J(5/n) =  J(n/5) &= \left\{
\begin{aligned}
 1,  \quad &n \equiv \pm 1 \mod 5 \\
 -1, \quad &n \equiv \pm 2 \mod 5 \\
\end{aligned}
 \right. \\
 \end{array}
\]

\subsection{Frobenius index}

In   number theory, the concept of the ``least quadratic non-residue mod
$p$"\ is widely used, that is, for the natural number $n$ find the
smallest positive $c$ such that $J(c/n)=-1$. In our case, a similar
but slightly different value is required.

\begin{definition} Let $n$ be an odd number and not a perfect square.
Its \textit{Frobenius index} $\ind_F(n) $ is the smallest $c$ among the
numbers $[-1,2,3,4,5,6, \dots]$ such that the Jacobi symbol $J(c/n)
\ne 1$.
\end{definition}

It follows from the multiplicativity of the Jacobi symbol that if a
Frobenius index is positive, then it is prime. 

It is not difficult to find out when the Frobenius index
$c=\ind_F(n)$ takes small values:

If $n\equiv 3 \mod 4$ then $c=-1$.

If $n\equiv 5 \mod 8$ then $c=2$.

Now we assume that $n$ is not divisible by $3$.

If $n\equiv 17 \mod 24$ then $c=3$. If ${n\equiv 1 \mod 24}$ then
$c\ge 5$.

Now we assume that $n$ is not divisible by $3$ and $5$.

If $n\equiv 73$ or $97 \mod 120$ then $c=5$. If $n\equiv 1$ or $49
\mod 120$ then $c\ge 7$.

\subsection{Quadratic field}

Let $c$ be a square-free integer and $z=a+b\sqrt{c} \in {\mathbb
Z}[\sqrt{c}\,]$. The number $a$ is called the \textit{rational part} of $z$,
$a=Rat(z)$, and $b$ is called the  \textit{irrational part}, $b=Irr(z)$. The number
$N(z)=a^2-b^2 c$ is called the \textit{norm} of $z$, $\overline{z} =
a-b\sqrt{c}\ $ is the \textit{conjugate} of $z$. So $N(z_1z_2)=N(z_1)N(z_2)$,
$N(z) = z \cdot \overline{z}$.

If $p$ is a prime and $J(c/p)=-1$ then the ring ${\mathbb Z}_p[\sqrt{c}\
]$ is isomorphic to the Galois field $GF(p^2)$. The map
\[
  z \rightarrow z^p \mod p
\]
is the Frobenius automorphism and $z^p \equiv \overline{z}$.

If $J(c/p)=+1$ then there exists $d \in {\mathbb Z}_p: d^2=c \mod p$. The
ring ${\mathbb Z}_p[\sqrt{c}\,]$ is isomorphic to the ${\mathbb Z}_p
\times {\mathbb Z}_p$ and the isomorphism is given by the formula:
\begin{equation} \label{eq_a1}
 a+b\sqrt{c} \rightarrow (a+bd, a-bd)\,.
\end{equation}
In this case $z^p \equiv z \mod p$.

\section{Frobenius primality test}\label{sec_Frob}

\subsection{Definition}

\begin{definition}\label{FrobTest}
Let $n$ be an odd number and not a perfect square, and let  $c = Ind_F(c)$ be the
Frobenius index. Let
\[
 z = \left\{
   \begin{array}{lll}
     2+\sqrt{c}, && c=-1, 2, \\
     1+\sqrt{c}, && c\ge 3. \\
   \end{array}
 \right.
\]
We call $n$ a \textit{Frobenius prime} if
\begin{equation}\label{def_Frob}
  z ^ n \equiv \overline{z} \mod n.
\end{equation}
\end{definition}

\begin{remark}
If $J(c/n)=0$, then $n$ is divided by $c$. This is a trivial case. So
we shall assume that $J(c/n)=-1$.
\end{remark}

The equality (\ref{def_Frob}) holds for any prime $n$ with
$J(c/n)=-1$.

If composite number $n$ is a Frobenius prime, then we call it a
\textit{Frobenius pseudoprime} (FPP). More precisely, if $z=a+b\sqrt{c}$ and
$z^n \equiv \overline{z} \mod n$, then the number $n$ will be called
\textit{Frobenius pseudoprime with parameters} $(a,b,c)$, or $FPP(a,b,c)$.

In other words, the FPP numbers are those on which the Frobenius
test is wrong.

\begin{example}
Let $n=19$, so $c=-1$, $z=2+i$,
\[
 z^n = -3565918 + 2521451 \cdot i \equiv 2-i \mod n.
\]
\end{example}

\begin{example}
Let $n=33$, so $c=-1$, $z=2+i$,
\[
 z^n \equiv 2 + 22 \cdot i \mod n \ne \overline{z}.
\]
\end{example}

\begin{example}
Let $n=17$, so $c=3$, $z=1+\sqrt{3}$,
\[
 z^n =  13160704+7598336 \sqrt{3} \equiv 1-\sqrt{3} \mod n.
\]
\end{example}

Note that if $n$ is $FPP(a,b,c)$, then $n$ is pseudoprime to a base
$N(z)=a^2-b^2c$, that is the Frobenius test includes the Fermat
test.

A comparison of the irrational part is actually a Lucas test. Thus,
the Frobenius test is a combination of the Fermat and Lucas tests.

\begin{hypothesis}
Frobenius pseudoprime numbers do not exist.
\end{hypothesis}

In other words, the Frobenius test is never wrong. It is also useless to seek a counterexample by a straightforward search. For as it will be proved in Proposition~\ref{prop_64} it is not among the numbers less than $2^{64}$.
It is more likely to find a FPP in the form of the product of primes.

\begin{remark}
The choice with the base $z=2+\sqrt{c}$ or $z=1+\sqrt{c}$ is not random.
For some $n$ may exist ``bad'' bases, or in the terminology of the
work \cite{DamFr2} ``liars''. The smallest example is $n=7\cdot
19\cdot 43=5719$. In this case the base $z=4689+\sqrt{-1}$ is
``liar'' that is
\[
 z^n = \overline{z} \mod n.
\]
\end{remark}

\begin{definition}
Let $n$ be a Frobenius pseudoprime with parameters $(a,b,c)$. The
prime factor $p$ of $n$ we call $\Phi$-positive, if $J(c/p)=+1$ and
$\Phi$-negative, if $J(c/p)=-1$.
\end{definition}

Each FPP has an odd number of $\Phi$-negative factors and arbitrary
number of $\Phi$-positive.

%===============================================================================
\subsection{First important theorem}

The following statement (in slightly different formulations) is
proved in~\cite{DamFr2,Seysen}.

\begin{theorem} \label{main_thm}
Let $n$ be an $FPP(a,b,c)$, $n=pq$ where $p$ is prime. Then

a) if $J(c/p)=-1$, then $z^q \equiv z \mod p$.

b) if $J(c/p)=+1$, then $z^q \equiv \overline{z} \mod p$.
\end{theorem}

\begin{proof}
Let $J(c/p)=-1$, then $z^p \equiv \overline{z} \mod p$. The number
$n$ is FPP, that is $z^{pq} \equiv \overline{z} \mod pq$, so
\[
 z^{pq}\equiv (z^p)^q  \equiv \overline{z}^q  \equiv \overline{z} \mod p,
\]
and
\[
 z^q\equiv z \mod p\,.
\]

Let $J(c/p)=+1$, then $z^p \equiv z \mod p$. The number $n$ is FPP,
so $z^{pq} \equiv \overline{z} \mod pq$ and
\[
 z^{pq}\equiv (z^p)^q  \equiv z^q  \equiv \overline{z} \mod p\,.
\]
\end{proof}

\begin{corollary} \label{cor21}
Let $n$ be an $FPP(a,b,c)$ and $p$ be a $\Phi$-negaive prime divisor, $Q = \ord(z \mod p)$. Then
\[
    n/p \equiv 1 \mod Q,
\]
\[
    n \equiv p \mod Q.
\]
\end{corollary}

\begin{corollary} \label{cor_main_th}
Let $z=a+b\sqrt{c} \in {\mathbb Z}$ and $z^q=a_q+b_q \sqrt{c} \in {\mathbb
Z}$ and $n$ be a $FPP(a,b,c)$, $n=pq$, where $p$ is prime. Then

a) if $J(c/q)=+1$ then $p$ is a prime factor of $gcd(a_q-a,
b_q-b)$;

b) if $J(c/q)=-1$ then $p$ is a prime factor of $gcd(a_q-a,
b_q+b)$.

\end{corollary}

\begin{example}
Let $q=31,\ c=5$. Then $J(c/q)=+1$ and
\[
 (1+\sqrt{c})^q = a_q+  b_q\sqrt{c} =3232337626136576+1445545331654656 \sqrt{c}
\]
and $gcd(a_q-a, b_q-b)=104005$, so $p$ is one of the prime factor of
$104005$: $5, 11, 31, 61$.
\end{example}

\begin{example}
Let $q=37,\ c=5$. Then $J(c/q)=-1$ and
\[
 (1+\sqrt{c})^q = 3712124497172627456+1660112543324045312 \sqrt{c}
\]
and $gcd(a_q-a, b_q+b)=37$, so so $p$ can be only $37$.
\end{example}

\begin{remark}
Although the numbers $ a_q, b_q $ grow rather quickly, the
corresponding common divisor are not too large and can be factorized
up to $q$ equal to many millions.
\end{remark}

%==============================================================================
\subsection{Multiple factors}\label{multiple}

\begin{theorem} \label{thm:mult_factors}
Let $p$ be a prime, $n=p^2q$ for some $q$ ($q$ can be a multiple of
$p$) and $n$ be a $FPP(a,b,c)$. Then
\[
 z^p \equiv \overline{z} \mod p^2\,.
\]
\end{theorem}

\begin{proof}
In the ring ${\mathbb Z}_{p^2}[\sqrt{c}\,]$:
\[
  (a+pb)^p \equiv a^p \mod p^2.
\]
So
\[
 z^{p^2q} \equiv \overline{z} \mod p^2q\,,
\]
and therefore
\[
 z^{p^2q} \equiv \overline{z} \mod p\,.
\]
As $z^{p^2} \equiv z \mod p$, so $z^q \equiv z^p \equiv \overline{z}
\mod p$ and
\[
 z^p \equiv \overline{z} +pu \mod p^2\,,
\]
\[
 z^q \equiv \overline{z} +pv \mod p^2
\]
for some $u,v \in {\mathbb Z}_{p}[\sqrt{c}\,]$. Then
\[
 z^{pq} \equiv (z^q)^p \equiv
 (\overline{z} +pv)^p \equiv \overline{z}^p \equiv z+p\, \overline{u} \mod p^2\,,
\]
\[
 z^{p^2q} \equiv (z^{pq})^p \equiv
 (z +p\overline{u})^p \equiv \overline{z}+pu \mod p^2\,.
\]

On the other hand $z^n\equiv \overline{z} \mod p^2$, that is $u=0$
therefore $z^p \equiv \overline{z} \mod p^2$.
\end{proof}

\begin{corollary}
If $n=p^2q$ is a $FPP(a,b,c)$, then $p^2$ is also  $FPP(a,b,c)$.
\end{corollary}

\begin{corollary}
    If $n=p^2q$ is a $FPP(a,b,c)$, then $N(z)^{p-1} \equiv 1 \mod p^2$,
    where $N(z)$ is a norm of $z$.
\end{corollary}

%==============================================================================
\subsection{$\Phi$-positive factor}
\label{f_pos}

There are very few such numbers (see section \ref{f_pos_calc}), but
they still exist.

\begin{theorem} \label{thm51}
Let $n$ be a Frobenius pseudoprime, $z = a + b \sqrt{c}$ and $p$ is
a $\Phi$-positive prime factor of $n$, $n=p\cdot q$, $c \equiv d^2
\mod p$. We introduce the notation:
\[
    z_1 = a+b \cdot d \mod p,
\]
\[
    z_2 = a-b \cdot d \mod p,
\]
$z_1, z_2 \in {\mathbb Z}_p$. Then
\begin{equation}\label{z1q}
    z_1^q \equiv z_2 \mod p,
\end{equation}
\begin{equation}\label{z2q}
    z_2^q \equiv z_1 \mod p,
\end{equation}
\end{theorem}

\begin{proof}
By definition:
\[
 (a + b \sqrt{c})^{pq} = a - b \sqrt{c} \mod p\,.
\]
If $J(c/p) = +1$ then $z^p = z$, so
\[
 (a + b \sqrt{c})^ q = a - b \sqrt{c} \mod p
\]
Using isomorphism ${\mathbb Z}_p[\sqrt{c}\,] \rightarrow {\mathbb Z}_p
\times {\mathbb Z}_p$, we obtain the required.
\end{proof}

\begin{corollary}
Let
\[
 N=z_1z_2 = a^2-b^2\cdot c
\]
and
\[
 w = z_1/z_2 = \frac{(a+bd)^2}{N} \mod p .
\]
Then
\[
 N^{q-1}=1\,,
\]
\[
 w^{q+1}=1\,.
\]
\end{corollary}

\begin{proof}
Multiplying equalities (\ref{z1q}) and (\ref{z2q}), we obtain
\[
 (z_1z_2)^q = z_1 z_2\,,
\]
or $N^{q-1}=1$, and dividing them into each other
\[
 (z_1/z_2)^q = z_2/z_1\,,
\]
or $w^{q+1}=1$.
\end{proof}

\begin{corollary}
Let $\alpha=\ord( N \mod\ p)$ and $\beta= \ord( w \mod p)$. Then
\[
 gcd(\alpha,  \beta ) \le 2.
\]
\end{corollary}

\begin{proof}
We have
\[
 q-1 = 0\mod \alpha,
\]
\[
 q+1 = 0\mod \beta.
\]
These two conditions can not be fulfilled simultaneously if $\alpha$
and $\beta$ have a common factor $>2$.
\end{proof}

\begin{corollary} \label{cor_plus}
Let $n$ be a Frobenius pseudoprime, $z = a + b \sqrt{c}$, $p$ be a
$\Phi$-positive prime factor of $n$ and $q=n/p$. Then
\[
 q \equiv A_p \mod M_p,
\]
where
\[
 M_p = lcm( \ord(z_1 \mod p),  \ord(z_2 \mod p)).
\]
\end{corollary}

\begin{proof}
If $q$ is increased by a multiple of $\ord (z_1 \mod p)$ and $\ord
(z_2 \mod p)$, then both sides of the equalities (\ref{z1q}) and
(\ref{z2q}) do not change.

Note that both $\ord(z_1 \mod p)$ and $\ord(z_2 \mod p)$ are divisors
of $p-1 $, so their least common multiple is also a divisor of
$p-1$.
\end{proof}

%===============================================================================
\subsection{$\Phi$-negative factor}\label{f_neg}

\begin{theorem} \label{coord_odd}
    Let $p$ be a $\Phi$-negative prime divisor of FPP $n$,
    that is $J(N(z), p)=-1$. Denote ${Q = \ord(z \mod p)}$. Then co-order
    $(p^2-1)/Q$ is odd. In particular, it follows that $Q \equiv 0 \mod 8$.
\end{theorem}
\begin{proof}
    As $z^p \equiv \overline{z} \mod p$, then $N(z) = z\cdot
    \overline{z} = z^{p+1}$, so $z^{(p^2-1)/2} = N(Z)^{(p-1)/2} = J(N(z),p)$.
    As $z^{(p^2-1)/2}\ne 1$, then co-order is odd.
\end{proof}

\begin{theorem} \label{admit_div}
Let $n$ be an FPP, $c=\ind_F(n)$ be its Frobenius index and $p$ be an $\Phi$-negative prime divisor of $n$.

a) If $c=-1$, then $p \equiv 3 \mod 4$.

b) If $c=2$, then $p \equiv 5 \mod 8$ and the product of all $\Phi$-negative prime divisors of  $n$
equals to $ 1 \mod 8$.

c) If $c=3$, then $p \equiv 17| 19 \mod 24$. In this case, there must be an odd number
of divisors equals  $17$ modulo $24$ (therefore, at least one is required).
There must be an even number of divisors $p_i$ which equals $19$ modulo $24$
and  $gcd(24, \ord(z \mod p_i))\le 2$.

d) If $c=5$, then $p \equiv 1 \mod 4$.

e) If $c=7$, then $(p \mod 24) < 12$ (that is $1|5|7|11$).

Denote $\ord(z \mod p)$ by $Q$ and $gcd(Q, 24)$ by $d$.

f) If $c\ge 5$, then $ p \equiv 1 \mod d$.

g) If $c\ge 5$ and $c'$ is a primve divisor of $Q$, $c'<c$, then
$J(p,c')=1$ (not $J(c',p)$, but $J(p,c')$).
\end{theorem}

\begin{proof}
Let $n=pq$ and  $Q$ be an order of $z$ modulo $p$.

a) by definition of $\Phi$-negative divisor.

b) as $\ind_F(n)=2$, then  $n \equiv 5 \mod 8$ and $z = 2 + \sqrt{2}$,
$N(z)=z \cdot \overline{z}=2$.

By Thm.~\ref{main_thm} we have $q \equiv 1 \mod Q$.
However according to Thm.~\ref{coord_odd} the number $Q$ is divisible by $8$
which means $q \equiv 1 \mod 8$. Therefore, $n=pq \equiv p \mod 8 =5$.

Thus, the product of all $\Phi$-negative prime divisors of $n$ is $1$ or $5$ modulo $8$.
Each $\Phi$-positive prime divisor equals to $\pm 1 \mod 8$, therefore
their product $\equiv \pm 1 \mod 8$. But $-1$ is impossible,
since in this case condition ${n\equiv 5 \mod 8}$ fails.

c) In this case $n \equiv 17 \mod 24$. Since $J(3/p)=-1$, $p
\equiv \pm 5 \mod 12$, that is ${p \equiv 5|7|17|19 \mod 24}$.
In this case $N(z)=N(1+\sqrt{3})=-2$. If $p \equiv 5|7 \mod 24$, then
$J(N(z),p)=-1$. It follows from Thm.~\ref{coord_odd} that in this case $Q$ is
divided by $8$, that is $q \equiv 1 \mod 8$. So $n=pq
\equiv p \mod 8 =5|7$, which contradicts the fact that $n\equiv 17 \mod
24$.

Let  $p \equiv 19 \mod 24$. Then
\begin{itemize}
    \item $q \equiv 11 \mod 24$.
    \item $q \equiv 1 \mod \ord(z,p)$.
\end{itemize}
Let $d=gcd(\ord(z,p),24)$. Therefore, $10
\equiv 0 \mod d$, which is only possible if $gcd(\ord(z,p),24)=2$ or $1$.

The statement about the number of multipliers follows from the fact that $n \equiv 17 \mod 24$
and  $g^2=1$ for all $g \in {\mathbb Z}_{24}^*$.

d) If $c=5$, then $z=1+\sqrt{5}$, $N(z)=1-5=-4$ and $J(N(z),p)=J(-1,p)$.
Assume that $p\equiv 3 \mod 4$, then $J(N(z),
p)=-1$ and according to the theorem (\ref{coord_odd}), $Q=\ord(z \mod p)$
is divided by $8$. Since $q \equiv 1 \mod Q$, then $n=pq \equiv 3 \mod 4$.
But Frobenius index $\ge 5$, so  $n \equiv 1 \mod 24$.

e) Since $z=1+\sqrt{7}$, then $N(z) = 1-7=-6$. Therefore
$J(N(z),n)=1$ if $n\equiv 1,5,7,11 \mod 24$ and $J(N(z),n)=-1$ if
$n\equiv 13,17,19,23 \mod 24$.

In second case $\ord(z,p)$ is divided by $8$ and congruence $pq \equiv 1
\mod 24$ are impossible.

Thus, if $c=7$, then all $\Phi$-negitive prime divisors of an
FPP must be congruenced $1,5,7$ or $11$ modulo
$24$.

f) All invertible residues $k$ modulo $24$ have a useful
property:  ${k^2 \equiv 1 \mod 24}$. So the congruence $pq \equiv 1
\mod 24$ can be rewritten as $q \equiv p \mod 24$. With the congruence
$q\equiv 1 \mod Q$ we get what we need.

g) By definition Frobenius index $J(c',n)$ must be equals $1$.
Since $n\equiv 1 \mod 24$ and ${J(q,c')=J(1+\alpha c', c')=1}$.
\[
    J(c',n) = J(n,c')  = J(pq,c') = J(p,c')J(q,c') = J(p,c')=+1.
\]
\end{proof}

\subsection{z-consistent prime factors}

\begin{theorem}\label{thm_coprime2}
Let $n$ be an FPP and $p_1,p_2$ its two $\Phi$-negative divisors.
If ${d =GCD( \ord(z,p_1), \ord(z,p_2))>1}$ then
\[
    p_1 \equiv p_2 \mod d
\]
\end{theorem}

\begin{proof}
From Corollary~\ref{cor21}, it follows
\[
    n \equiv  p_1 \mod \ord(z,p_2),
\]
\[
    n \equiv  p_2 \mod \ord(z,p_1).
\]
Therefore
\[
    n \equiv  p_1  \equiv  p_2 \mod d.
\]
\end{proof}

\begin{theorem}\label{thm_coprime1}
Let $n$ be an $FPP(a,b,c)$, $p$ is $\Phi$-negative prime divisor and
$Q = \ord(z \mod p)$. Then $Q$ and $n$ are coprime.
\end{theorem}

\begin{proof}
We need to prove that $Q$ is not divisible by any prime divisor of $n$,
including $p$. The number $Q$ is a divisor of $p^2-1$ and, therefore, is not divisible by $p$.

According to Corollary~\ref{cor21} $n/p \equiv 1 \mod Q$, so $Q$
is coprime with $n/p$, hence $Q$
is coprime with each of its prime divisors.
\end{proof}

\begin{definition}
A pair of primes divisors $(p_1, p_2)$ of FPP $n$ is called $z$-consistent if:
\[
    \begin{array}{l}
        J(c/p_1)=-1\\
        J(c/p_2)=-1\\
        \ord(z,p_1) \ne 0 \mod p_2 \\
        \ord(z,p_2) \ne 0 \mod p_1 \\
        p_1 \equiv p_2 \mod GCD( \ord(z,p_1), \ord(z,p_2)).\\
    \end{array}
\]
\end{definition}

Thus, all $\Phi$-negative FPP divisors are pairwise
consistent.

Let $n$ be FPP and $p$ its $\Phi$-positive prime factors. Corollary~\ref{cor_plus}
implies 
%can be written in the form
\[
 n \equiv D_p \mod M_p,
\]
for some $D_p, M_p$.

If $p$ is a $\Phi$-negative prime factor $n$, according to the main Thm.~\ref{main_thm}
\[
 q \equiv 1 \mod \ord(z \mod p)
\]
or
\[
 n \equiv D_p \mod M_p,
\]
where $D_p=p$, $M_p = \ord(z \mod p)$.

Let $p_1, p_2$ are two different prime factors of FPP $n$,
$\Phi$-positive of negative and $n=p_1p_2q$. So
\[
 n \equiv D_{p_1} \mod M_{p_1},
\]
\[
 n \equiv D_{p_2} \mod M_{p_2}.
\]
From this it follows that in this case we have
\begin{equation}\label{soglas}
 D_{p_1} \equiv D_{p_2} \mod gcd( M_{p_1}, M_{p_2}).
\end{equation}
This relation does not depend on $q$, only on $p_1$ and $p_2$.

\begin{definition}
Given $z \in {\mathbb Z}[\sqrt{c}]$. Two primes will be called
\textit{$z$-consistent} or simply \textit{consistent}  if the relation (\ref{soglas})
holds for them.
\end{definition}

\begin{theorem} \label{thm_last}
Let $n$ be a Frobenius pseudoprime. Then all its prime factors are
pairwise consistent.
\end{theorem}

%===============================================================================
\section{Results of calculations}
\label{sec:calc}
The hypothesis asserting that there are no Frobenius pseudoprimes (FPP)
can not yet be proved.
Below are related results that we were able
to establish.
%Consider what we managed to do in this
%direction.

%------------------------------------------------------------------------------
\subsection{Search for small FPP}

We considered all composite odd numbers that are not complete squares. All such numbers up to $ 350 \cdot 10^{9} $ were checked
on being a FPP.
This computation took 
few days on a standard PC (Intel(R) Pentium(R) CPU G4500 @3.50GHz). As the result we have the following proposition.

\begin{proposition}\label{prop_1direct}
    There is no FPP less than $350$ billions.
\end{proposition}

%------------------------------------------------------------------------------
\subsection{Search for large FPP with a large Frobenius index}

As it was mentioned above, if ${n\equiv 1 \mod 24}$  then
$\ind_F(n)\ge 5$. The Frobenius index can be arbitrarily large. Among
the numbers $<2^{32}$, the largest value of the index is $101$ and it is for 
the number $2805\,44\,681$.
In~\cite{khash2015} a complete list of $458\,069\,912$
numbers less than $2^{64}$, whose index of Frobenius $>128$ is obtained.
All these numbers are not FPP. As the result we have the following proposition.
\begin{proposition}\label{prop_2large_ind}~\cite{khash2015}
There is no FPP less than $2^{64}$ with the Frobenius index larger than $128$.
\end{proposition}

%------------------------------------------------------------------------------
\subsection{Search for large FPP with multiple factors}

Sec.~\ref{multiple} contains proofs of the properties that should be satisfied by multiple prime factors of FPP. A direct calculation of these
properties showed that FPP does not have multiple factors less than
$2^{32}$ with the Frobenius index $c<128$ (without restriction on the value of FPP).
The total computation time (with 3.50GHz) is about two days. As the result we have the following proposition.

\begin{proposition}\label{prop_3multipl}
There are no FPPs smaller than $2^{64}$ having multiple prime
factors.
\end{proposition}

%------------------------------------------------------------------------------
\subsection{Estimation of the product of all factors except one (for FPP)}

We propose the following
idea to significantly simplify the search for FPP. 

Let $n$ be $FPP(a,b,c)$ and $p$ the prime factor of $n$, $q=n/p$. In
this case $z=a+b\sqrt{c} \in {\mathbb Z}$ and $z^q=a_q+b_q \sqrt{c}$.
Corollary~\ref{cor_main_th} implies that
for every $q$ there is a small number of possible $p$, as $p$ has to be a divisor of $D=gcd(a_q-a, b_q \pm b)$, where the
sign "$+$"\ or "$-$"\ is taken depending on the sign a $J(c/q)$. 

\textit{In practice it turned out that the number of possibilities for $p$ is not just small but very small: about $1$ to $3$ different $p$.}  

Thus, for a fixed $z=a+b\sqrt{c}$, for each positive $q$ we perform
the following steps:

\begin{enumerate}
\item calculate $z^q=a_q+b_q \sqrt{c}$,
\item calculate $D=gcd(a_q-a, b_q \pm b)$,
\item prime factorization of $D$: $D=p_1\dots p_s$,
\item for each $p_i$ check whether $n_i=q\cdot p_i$ is FPP.
\end{enumerate}

If $q$ is of the order of several million, then $a_q, b_q$ will have
a length of up to tens of millions of bits. However, the number $D$
in all cases will not be so large and, most importantly, is
decomposed into small prime factors.

Within a reasonable time (hours) the result is as follows:

\begin{proposition}\label{prop_6_q}
Let $n$ be an FPP (any size, not necessarily $<2^{64}$) with an
Frobenius index $c=\ind_F(n)<128$. Then $n$ has no prime factors $p$
such that $n/p<2^{21}$.
\end{proposition}

%------------------------------------------------------------------------------
\subsection{A complete list of $\Phi$-positive prime factors  less than $2^{32}$ for a FPP}\label{f_pos_calc}

In Sec.~\ref{f_pos} properties of the $\Phi$-positive
factors $p$ of FPP $n=pq$ are proved and an algorithm for finding
numbers possessing these properties is proposed. This algorithm gives us the possible $\Phi$-positive prime factors
$p$ and some congruence relation for $q$:
\[
 q \equiv q_p \mod A_p\
\]
for a given $p$. An additional constraint comes from the congruence relation implied by the Frobenius index:
\[
  \begin{array}{ll}
    n \equiv 3  \mod  4, & if\  \ind_F(n)=-1, \\
    n \equiv 5  \mod  8, & if\  \ind_F(n)=2, \\
    n \equiv 17 \mod 24, & if\  \ind_F(n)=3, \\
    n \equiv 1  \mod 24, & if\  \ind_F(n)\ge 5, \\
  \end{array}
\]
and if $\ind_F(n)\ge 5$ then $J(c/n)=+1$ for all $c<\ind_F(n)$.

There are few such numbers $p$. For $c=\ind_F(n)<128$ and $p<2^{32}$
we have only $26$ numbers:
\[
  \begin{array}{|rr|rr|rr|rr|}\hline
 c  &          p &   c &         p  &   c &        p   &   c &        p \\\hline
 -1 &    2276629 &  11 &      98641 &  61 &       271  &  89 & 109000877\\
 -1 &   30906409 &  17 &     125597 &  67 &     75011  &  89 & 136973443\\
 -1 &  806361541 &  23 &    5966803 &  67 &  25742443  & 101 &       137\\
  2 &       8191 &  29 &      12637 &  83 &      1931  & 103 &      6863\\
  2 & 2147483647 &  31 & 3596719249 &  83 &   3278741  & 103 &3523679801\\
  7 &         31 &  43 &     329947 &  83 & 806898559  & 107 & 219920461\\
  7 &       3923 &     &            &     &            & 127 & 713342911\\
  \hline
  \end{array}
\]

If we assume that $n=pq<2^{64}$ then most of these $n$ can be
directly checked whether they are a FPP or not. After this, only the following eight numbers
remain, for which a direct verification is still difficult (too time-consuming):
\[
\begin{array}{|rr|rr|rr|rr|}\hline
    c  &          p &   c &         p  &   c &      p   &   c &      p \\\hline
    2 &       8191 &   7 &       3923 &  29 &    12637 &  83 &3278741 \\
    7 &         31 &  11 &      98641 &  61 &       271& 101 &    137 \\
    \hline
\end{array}
\]

Note that in the case of a large Frobenius index, the computation can be significantly
reduced if you do not iterate over all numbers that are multiples of $p$,
but only over those for which the Frobenius index is equal to the given $c$ (as given in the table above).
After that, only the following list of five $\Phi$-positive divisors remains unchecked:
\[
\begin{array}{|rr|rr|rr|}\hline
    c &          p &   c &      p   &   c &    p \\\hline
    2 &       8191 &   7 &    3923  & 101 &  137 \\
    7 &         31 &  61 &     271  &     &      \\
    \hline
\end{array}
\]
We see an FPP $n$ such that $n<2^{64}$ has  
two $\Phi$-positive factors less than
$2^{32}$ only if its Frobenius index $\ind_F$ is $7$. That is  $z=1+\sqrt{7}$, and these factors are $31$
and $3923$. By direct verification within a reasonable time (several
hours), one can make sure that both factors can't occur
simultaneously. As the result we have the following statement.
\begin{proposition}\label{prop_4pos}
$\Phi $-positive prime factors less than $2^{32}$ for FPPs smaller
than $2^{64}$ can be only $5$ numbers mentioned above, and two such
factors can not meet simultaneously.
\end{proposition}

%------------------------------------------------------------------------------
\subsection{The main proposition: there are no FPP less than $2^{64}$}

Let $n$, $n<2^{64}$ be an FPP. Below it a summary of what we have discovered so far for such numbers: 

a) $n>350 \cdot 10^9$ (Proposition (\ref{prop_1direct})).

b) Frobenius index $c=\ind_F(n)<128$ (Proposition (\ref{prop_2large_ind})).

c) $n$ does not have multiple factors (Proposition (\ref{prop_3multipl})).

d) The product of all prime factors except one is greater then $2^{21}$ (Proposition (\ref{prop_6_q})).

e) $\Phi$-positive factors $p$ may be only for $c=2\, (p=8191)$,
$c=7\, (p=31, 3923)$,  $c=61\, (p=271)$,  $c=101\, (p=137)$(Proposition (\ref{prop_4pos})).

Later in this section, we assume that FPP $n$ satisfies all these conditions.

\begin{proposition}\label{prop_1factor}
Let $n<2^{64}$ be an FPP. Then $n$ does not have prime factors from
the interval $(50159, 2^{32})$.
\end{proposition}

\begin{proof} The absence of $\Phi$-positive factors of this size proved earlier.
    Therefore, we consider only $\Phi$-negative factors.

    Let $n<2^{64}$ be a FPP with $z=a+b\sqrt{c}$, $c<128$ and $p$ be a
    prime factor of $n$, $J(c/p)=-1$. We denote $n/p$ by $q$. According
    to Thm.~\ref{main_thm}
    \[
    z^{q-1} \equiv 1 \mod p,
    \]
    that is
    \[
    q \equiv 1 \mod {\rm \ord}(z \mod p).
    \]
    or
    \[
    q = 1 + k Q_p
    \]
    for some $k\ge 1$, where $Q_p = {\rm \ord}(z \mod p)$. As
    $n=pq<2^{64}$, then $q < 2^{64}/p$. Hence, we find the restriction
    on $k$: $k\le k_{max}$. This means that the only valid candidates
    for the FPP will be in the numbers
    \[
    p(1+Q_p), p(1+2Q_p), \dots, p(1+k_{max}Q_p).
    \]
    As a result, in a reasonable time (a few hours for a fixed Frobenius
    index) you can check all $\Phi$-negative number in the interval
    $(2^{17},2^{32})$.
\end{proof}

\begin{example}
    Let $z=2+i$, $p=10\,000\,019$. Then $Q_p = 1\,666\,730\,000\,060 =
    (p^2-1)/6$ and for any $k\ge 1$ we have  $n=pq>2^{64}$. That is for this
    $p$ there is no suitable $q$.

    Let $p=1\,000\,003$. Then $Q_p = 1\,000\,006\,000\,008 = p^2-1$ and
    inequality $n=pq<2^{64}$ holds for $k\le 18$. That is the only suitable values for $q$
    are
    \[
    1+Q_p\,, 1+2Q_p\,, \dots \,, 1+18Q_p\,.
    \]
    It is easy to check that for all these values of $q$, the number $n=pq$ is not an
    FPP, that is $p$ cannot be a divisor of an FPP
    that is less than $2^{64}$.

    Let $p=100\,003$. Then $Q_p = 434\,808\,696 = (p^2-1)/23$ and
    inequality $n=pq<2^{64}$ holds for $k\le 424\,236$.
    With the smaller $p$ 
    the computation time quickly increases. 
    Verification of
    all eligible $q$ in this case takes already several minutes.
\end{example}

By a somewhat larger search, it is possible to construct for each
index $c<128$ a complete list of possible  $\Phi$-negative
prime factors of FPP. For example, for $c=-1$ ($z=2+i$) the list
will consist of $350  $ prime numbers:
\[
7, 11, 19, 23, 31, 43, 47, \dots, 39439,\ 50159.
\]

\[
\begin{array}{rrr|rrr}
    \ind_F   &   \mbox{The number}& {\rm max}(p_i) & \ind_F   &   \mbox{The number} & {\rm max}(p_i) \\
    &   \mbox{of primes }&                &         &   \mbox{of primes}&                \\
    -1  &  350  & 50159 &   53   &   39 & 21841 \\
    2  &   91  & 33461 &   59   &   40 &  5651 \\
    3  &   72  & 23057 &   61   &   39 & 17749 \\
    5  &  105  & 49477 &   67   &   31 &  5557 \\
    7  &   48  & 47791 &   71   &   30 & 34501 \\
    11  &   50  &  9437 &   73   &   34 &  6883 \\
    13  &   63  & 19141 &   79   &   32 &  9041 \\
    17  &   50  &  8681 &   83   &   33 & 38669 \\
    19  &   38  & 25939 &   89   &   26 &  7867 \\
    23  &   44  &  8069 &   97   &   32 &  6221 \\
    29  &   46  &  7687 &  101   &   46 &  9901 \\
    31  &   37  & 38917 &  103   &   30 & 14341 \\
    37  &   55  & 15289 &  107   &   21 &  8539 \\
    41  &   39  & 19447 &  109   &   38 & 13001 \\
    43  &   37  & 15277 &  113   &   37 & 13241 \\
    47  &   36  & 12109 &  127   &   22 &  4987 \\
\end{array}
\]

\begin{corollary}\label{cor_2factor}
    Let $n<2^{64}$ be an FPP. Then $n$ has more then two prime factors.
\end{corollary}
\begin{proof}
    If $n$ has exactly two prime factors, then the smallest of them by the Proposition~\ref{prop_1factor} should not be more
    than $50159$, which contradicts  Proposition~\ref{prop_4pos}.
\end{proof}

\begin{proposition}\label{prop_2factors}
    Let $n<2^{64}$ be an FPP and $p_1, p_2$ be its prime factors, both
    less $2^{32}$. Then $p_1p_2<2^{17}$. Moreover, for each $c<128$, we
    have a complete list of possible pairs $(p_1,p_2)$:
    \[
    \begin{array}{rrr|rrr}
        \ind_F   &   \mbox{The number}& {\rm max}(p_1p_2) & \ind_F   &\mbox{The number} & {\rm max}(p_1p_2) \\
        &   \mbox{of pairs}  &                  &        &\mbox{of pairs   }&                \\
        -1      &    184    &128929 &   53      &      2    & 53947 \\
        2      &     64    & 28345 &   59      &      2    & 29857 \\
        3      &     36    & 40681 &   61      &      0    &     - \\
        5      &     56    & 58669 &   67      &      1    & 66667 \\
        7      &     11    & 24641 &   71      &      0    &     - \\
        11      &      8    & 42127 &   73      &      0    &     - \\
        13      &     31    & 42199 &   79      &      0    &     - \\
        17      &     22    & 77981 &   83      &      0    &     - \\
        19      &      0    &     - &   89      &      1    & 58277 \\
        23      &      1    & 24461 &   97      &      1    & 29651 \\
        29      &      1    & 53947 &  101      &      0    &     - \\
        31      &      5    & 34103 &  103      &      0    &     - \\
        37      &      7    & 58969 &  107      &      0    &     - \\
        41      &      0    &     - &  109      &      0    &     - \\
        43      &      0    &     - &  113      &      0    &     - \\
        47      &      4    &103351 &  127      &      0    &     - \\
    \end{array}
    \]
\end{proposition}
\begin{proof}
    Suppose that an FPP $n$ has two factors of $p_1$ and $p_2$ less than
    $2^{32}$. Then both $p_1$ and $p_2$ should be contained in a
    relatively small list
    which is constructed using Proposition~\ref{prop_1factor}.

    Factors need to be 
    $z$-consistent
    and for $q=n/(p_1p_2)$ the following congruence
    relations should hold:
    \[
    q \equiv D_{p_{12}}  \mod gcd(M_{p_1},M_{p_2})
    \]
    for some $D_{p_{12}}$.

    Taking into account that $n=p_1p_2q < 2^{64}$, it often turns out
    that for a given pair $(p_1,p_2)$
    all possible $q$ are small and all
    corresponding $n$ can be
    thus easily checked whether they are an FPP or not. However, if     $(p_1,p_2)$ are small in a sense then
    the number of possible  $q$s is too
    large and we cannot check all of the corresponding $n$ on being 
    an FPP, and these are listed in the table above.
    (These remaining pairs will be addressed below).
\end{proof}

\begin{remark}
    Among these pairs, there are none containing $\Phi$-positive numbers.
    In particular, an FPP $n$ does not have $\Phi$-positive factors less
    than $2^{32}$.
\end{remark}

\begin{corollary}\label{cor_3factor}
    Let $n<2^{64}$ be an FPP. Then $n$ has more than three prime
    factors.
\end{corollary}
\begin{proof}
    If $n$ has exactly three prime factors, at least two of them are
    less $2^{32}$ and according   Proposition~\ref{prop_2factors}
    their product is less than $128929$, which contradicts 
    Proposition~\ref{prop_4pos}.
\end{proof}

\begin{proposition}\label{prop_3factors}
Let $n<2^{64}$ be an FPP and $p_1, p_2, p_3$ be its prime factors less than $2^{32}$.
Then $c=-1$ and triple  $(p_1, p_2, p_3)$ is one of the following:
\[
\begin{array}{rrr}
    p_1 & p_2 & p_3 \\
    199 &   19 &   7 \\
    191 &  127 &  31 \\
    191 &   71 &  11 \\
     79 &   31 &  19 \\
     79 &   19 &   7 \\
     71 &   47 &  11 \\
\end{array}
\]
\end{proposition}

\begin{proof}
Pairs $(p_1, p_2)$, $(p_1, p_3)$, $(p_2, p_3)$ must
be present in the list of valid pairs given in Proposition~\ref{prop_2factors}. There are very few such triples.
For almost all triples all their possible multiples $n=p_1p_2p_3q$ can be checked on being an FPP in a
short time (hours).
Only those triplets that are specified in the statement of Proposition~\ref{prop_3factors}  are remained as a possibility.
\end{proof}

We have already established in Corollaries~
\ref{cor_2factor} and~\ref{cor_3factor} that 
an FPP $n$, $n<2^{64}$
has more than
two and than more than three prime factors. 
\begin{corollary}\label{cor_4factor}
    Let $n<2^{64}$ be an FPP. Then $n$ has more than four prime factors.
\end{corollary}
\begin{proof}
If $n$ has exactly four prime factors, at least three of them are
less $2^{32}$ and by Proposition~\ref{prop_4pos} their
product is greater than $2^{21}$. However, then for all triplets in Proposition~\ref{prop_3factors} the product
$p_1 p_2 p_3$ is less than $2^{21}$.
\end{proof}

\begin{proposition}\label{prop_64}
There are no FPP  less than $2^{64}$.
\end{proposition}
\begin{proof}
By Corollary~\ref{cor_4factor}, an FPP $n$ has at least four prime factors $<2^{32}$.
Each triple of these four must be present in the list of Proposition~\ref{prop_3factors}.
But they are not there.
\end{proof}

%------------------------------------------------------------------------------
\section{An FPP cannot be a product  of small factors}
\label{sec:last}
\begin{proposition}\label{small_div}
Let $n$ be an FPP. Then $n$ has a prime divisor larger than $3000$.
\end{proposition}

To verify this statement, for each Frobenius index $c<3000$ , we iterate over all subsets
of valid prime factors, and they must all be pairwise consistent.

\begin{remark}
In fact, the lower bound $3000$ given in Proposition~\ref{small_div} 
can be improved for each
$c$. Below is the list of obtained lower bounds.
\[
\begin{array}{rr|rr|rr|rr}
    \ind_F &\mbox{border}& \ind_F & \mbox{border} & \ind_F & \mbox{border} & \ind_F & \mbox{border} \\
    -1  & 3067 &  53 & 4513  & 131 & 5897 & 223 & 6073 \\
     2  & 3109 &  59 & 4177  & 137 & 5209 & 227 & 5881 \\
     3  & 3089 &  61 & 4909  & 139 & 5881 & 229 & 5849 \\
     5  & 3793 &  67 & 5077  & 149 & 6217 & 233 & 5441 \\
     7  & 4177 &  71 & 5573  & 151 & 5113 & 239 & 6661 \\
    11  & 3637 &  73 & 4273  & 157 & 6829 & 241 & 5857 \\
    13  & 3049 &  79 & 5449  & 163 & 7057 & 251 & 6637 \\
    17  & 3361 &  83 & 5449  & 167 & 6449 & ... & ...  \\
    19  & 4649 &  89 & 5189  & 173 & 4937 &2971 & 7537 \\
    23  & 3251 &  97 & 6003  & 179 & 6361 &2999 & 9293 \\
    29  & 3361 & 101 & 4253  & 181 & 5209 &  & \\
    31  & 3733 & 103 & 6217  & 191 & 6823 &  & \\
    37  & 3169 & 107 & 6037  & 193 & 3469 &  & \\
    41  & 3529 & 109 & 4657  & 197 & 5449 &  & \\
    43  & 3677 & 113 & 4789  & 199 & 6361 &  & \\
    47  & 4273 & 127 & 6569  & 211 & 6121 &  & \\
\end{array}
\]
\end{remark}

\begin{remark}
These lower bounds depends 
only on our computational  capabilities (within a few hours of processor time).
Unfortunately, the volume of computations is growing exponentially,
so it is not possible to significantly improve these bounds,  even with the increase of the computation time.
\end{remark}

%===============================================================================
\section{Conclusions}

The FPP numbers are those on which the Frobenius test is fail.

\begin{hypothesis}
There are no Frobenius pseudoprime numbers.
\end{hypothesis}

Below are the facts about FPP that are known to date along with some new facts established in the present paper.

\begin{itemize}
\item The complexity of the Frobenius test is about twice that of Fermat or Miller-Rabin.

\item There are no examples of FPPs.

\item There are no FPPs less than  $2^{64}$.

\item Each FPP has a prime factor larger than $3000$.

\item Frobenius test is one of the most efficient primality tests to date!

\end{itemize}

\end{document}